\documentclass[10pt,letterpaper]{article}
\usepackage{amsthm,amsmath,amssymb}

\usepackage[letterpaper]{geometry}                
\usepackage{pdflscape}               
\usepackage{xcolor}			

\usepackage{mathtools} 

\usepackage{graphicx}
\usepackage{epstopdf}
\usepackage[colorlinks=true]{hyperref}

\usepackage{amsrefs}
\newcommand{\footremember}[2]{%
    \footnote{#2}
    \newcounter{#1}
    \setcounter{#1}{\value{footnote}}%
}

\newcommand{\NN}{\ensuremath{\mathbb N}}

\newcommand{\abs}[1]{\left| {#1} \right|}
\DeclareMathOperator{\gen}{gen}
\DeclareMathOperator{\Span}{span}

\theoremstyle{plain} 
\newtheorem{theorem}{Theorem} 

\newtheorem{proposition}[theorem]{Proposition}

\newtheorem{corollary}[theorem]{Corollary}
\newtheorem{conjecture}[theorem]{Conjecture}

\newtheorem{problem}[theorem]{Problem}
\theoremstyle{remark}
\theoremstyle{definition}
\newtheorem{definition}[theorem]{Definition}

\title{Counting Dope Matrices}
\author{%
    Noga Alon\footremember{Princeton}{Department of Mathematics, Princeton University, Princeton, NJ 08540, USA and Schools of Mathematics and Computer Science, Tel Aviv University, Tel Aviv 69978, Israel}
    \and Noah Kravitz\footremember{alsoPrinceton}{Department of Mathematics, Princeton University, Princeton, NJ 08540, USA}
    \and Kevin O'Bryant\footremember{CUNY}{City University of New York (College of Staten Island and The Graduate Center), Staten Island, NY 10314, USA}}
\date{\today}

\begin{document}
    \maketitle

\begin{abstract}
For a polynomial $P$ of degree $n$ and an $m$-tuple $\Lambda=(\lambda_1,\dots,\lambda_m)$ of distinct complex numbers, the dope matrix of $P$ with respect to $\Lambda$ is $D_P(\Lambda)=(\delta_{ij})_{i\in [1,m],j\in[0,n]}$, where $\delta_{ij}=1$ if $P^{(j)}(\lambda_i)=0$, and $\delta_{ij}=0$ otherwise. Our first result is a combinatorial characterization of the $2$-row dope matrices (for all pairs $\Lambda$); using this characterization, we solve the associated enumeration problem.  We also give upper bounds on the number of $m\times(n+1)$ dope matrices, and we show that the number of $m \times (n+1)$ dope matrices for a fixed $m$-tuple $\Lambda$ is maximized when $\Lambda$ is generic.  Finally, we resolve an ``extension'' problem of Nathanson and present several open problems.
\end{abstract}

\section{Introduction}
Let $P^{(j)}$ denote the $j$-th derivative of the polynomial $P(x)$.  Given a polynomial $P \in \mathbb{C}[x]$ of degree $n$ and an $m$-tuple $\Lambda=(\lambda_1, \ldots, \lambda_m)$ of distinct complex numbers, the \emph{dope matrix of $P$ with respect to $\Lambda$} is the $m\times(n+1)$ matrix $D_P(\Lambda)=(\delta_{ij})_{i \in [1,m], j \in [0,n]}$, where
    $$\delta_{ij} = \begin{cases} 1 &\text{if } P^{(j)}(\lambda_i) = 0 \\ 0 & \text{otherwise}\end{cases}.$$
In other words, the dope matrix records the ``zero pattern'' of the derivatives of $P$ evaluated on $\Lambda$.\footnote{Hence ``dope'' is an acronym for ``derivatives of a polynomial, evaluated''.}  We say that a $\{0,1\}$-matrix is simply \emph{dope} if it equals $D_P(\Lambda)$ for some $P$ and $\Lambda$.  For an $m$-tuple $\Lambda$ and a natural number $n$, let $$\mathcal{D}_n(\Lambda)=\{D_P(\Lambda): \deg(P)=n\}$$
denote the set of all $m \times (n+1)$ dope matrices arising from $\Lambda$.  Further, let
$$\mathcal{D}_n^m=\{D_P(\Lambda): \deg(P)=n, \Lambda\in \mathbb{C}^m \text{ distinct}\}$$
denote the set of all $m \times (n+1)$ dope matrices (where now we let both $P$ and $\Lambda$ vary).  There are many natural questions about the sets $\mathcal{D}_n(\Lambda)$ and how they relate to one other.

It is straightforward to check (see~\cite{Nathanson}) that if $\Lambda=\{\lambda_1, \ldots, \lambda_m\}$ and $\Lambda'=(\lambda'_1, \ldots, \lambda'_m)$ are related by $\lambda'_i=a+b\varphi(\lambda_i)$ where $a \in \mathbb{C}$, $b \in \mathbb{C}^{\times}$, and $\varphi$ is a $\mathbb{Q}$-automorphism of $\mathbb{C}$, then $\mathcal{D}_n(\Lambda)=\mathcal{D}_n(\Lambda')$.  In particular, if $\Lambda$ has size $1$ or $2$, then $\mathcal{D}_n(\Lambda)$ does not depend on the particular choice of $\Lambda$.  If $\Lambda$ has size at least $3$, then we lose nothing by assuming that (say) $\lambda_1=0$ and $\lambda_2=1$.

Nathanson~\cite{Nathanson}*{Theorem 2} noted that a $1$-row $\{0,1\}$-matrix is dope if and only if the last entry is $0$; in particular, the number of $1 \times (n+1)$ dope matrices is $2^n$. Our first result provides a necessary and sufficient condition for a $2$-row matrix to be dope.
\begin{theorem}\label{thm:m=2}
A $2 \times (n+1)$ $\{0,1\}$-matrix is dope if and only if for each $k\in[0,n]$, there are at most $k$ nonzero entries in the last $k+1$ columns.
\end{theorem}
In particular, for each pair $\Lambda$ of distinct complex numbers, the set $\mathcal{D}_n(\Lambda)$ consists of precisely the matrices satisfying the condition described in the theorem.  This characterization also allows us count the $2$-row dope matrices.
\begin{corollary}\label{cor:counting}
The number of $2 \times (n+1)$ dope matrices is $\binom{2n+1}{n}$. 

The number of $2 \times (n+1)$ dope matrices with exactly $n$ nonzero entries is $\frac{1}{n+2} \binom{2n+2}{n+1}$, the $(n+1)$-th Catalan number.  

The number of $2 \times (n+1)$ dope matrices modulo the action of swapping the two rows is $\frac{1}{2}\left[\binom{2n+1}{n}+\binom{n}{\lfloor n/2\rfloor} \right]$.
\end{corollary}
\noindent
The trivial upper bound on $|\mathcal{D}_n^m|$, the number of $m\times(n+1)$ dope matrices, is $2^{nm}$. The first statement of Corollary~\ref{cor:counting} immediately gives an improvement on this bound.
\begin{corollary}\label{cor:counting big}
Let $m \geq 2$, $n\geq 1$. The number of $m\times(n+1)$ dope matrices is at most $$\binom{2n+1}{n}^{m/2};$$ an upper bound for this quantity is $\big(\frac{2}{\sqrt{\pi n}}\big)^{m/2} 2^{mn}$.
\end{corollary}

Using known results~\cite{RBG} on zero patterns of polynomials, we also obtain the following upper bound.
\begin{theorem}\label{thm:zero patterns}
Let $m \geq 3$, $n \geq 2$.  The number of $m \times (n+1)$ dope matrices is at most $$\binom{3n+m-2+(m-2)n(n+1)/2}{m+n-2};$$ 
an upper bound for this quantity is $\binom{mn^2}{m+n}$, which in turn is at most $2^{(m+n) \log_2 (emn)}$.
\end{theorem}

A back-of-the-envelope calculation indicates that Corollary~\ref{cor:counting big} is stronger than Theorem~\ref{thm:zero patterns} if $m\le 10$ or $n\le 7$. Asymptotically, Theorem~\ref{thm:zero patterns} is stronger than Corollary~\ref{cor:counting big} when $n$ is smaller than around $2^m/m$ (e.g., $n=m$), and Corollary~\ref{cor:counting big} is stronger than Theorem~\ref{thm:zero patterns} when $n$ is larger than around $2^m/m$ (e.g., fixed $m$).

It is also of interest to estimate $|\mathcal{D}_n(\Lambda)|$ for fixed $\Lambda$, and the argument of Theorem~\ref{thm:zero patterns} gives a nontrivial upper bound in this setting.  For $0<p<1$, let $H(p)=-p\log_2(p)-(1-p)\log_2(1-p)$ denote the usual binary entropy function.

\begin{theorem}\label{thm:zero patterns2}
Let $m \geq 3$.  For each fixed $m$-tuple $\Lambda$ of distinct complex numbers, we have $|\mathcal{D}_n(\Lambda)|\leq \sum_{i=0}^n \binom{mn}{i};$
an upper bound for this quantity is $2^{mnH(1/m)}$.
\end{theorem}

In a related direction, for $m \geq 3$ it is natural to ask which $m$-tuples $\Lambda$ maximize $|\mathcal{D}_n(\Lambda)|$.  The set $X$ of $m$-tuples $\Lambda$ in $\mathbb{C}^m$ whose entries satisfy a nontrivial integer polynomial relation is a meager set (that is, a countable union of nowhere-dense sets), so we will call $\Lambda$  \emph{generic} if it lies outside of $X$.  (A random tuple $\Lambda$ chosen in any reasonable way will thus be generic.)  Recall that if $\Lambda'=\varphi(\Lambda)$ for some $\mathbb{Q}$-automorphism $\varphi$ of $\mathbb{C}$, then $\mathcal{D}_n(\Lambda)=\mathcal{D}_n(\Lambda')$.  In particular, generic $m$-tuples $\Lambda$ and $\Lambda'$ are related in this fashion: Since the entries of $\Lambda$ and $\Lambda'$ are algebraically independent, standard facts about transcendence bases and the Axiom of Choice show that desired automorphism of $\mathbb{C}$ exists.\footnote{We remark that the reliance on the Axiom of Choice can be eliminated.  Indeed, given a particular polynomial $P$ of degree $n$, let $K$ denote the extension of $\mathbb{Q}$ formed by adjoining the elements of $\Lambda \cup \Lambda'$ and the coefficients of $P$; without using the Axiom of Choice, we can find an automorphism $\varphi$ of $K$ that sends each $\lambda_i$ to $\lambda'_i$, and this suffices to show that $D_P(\Lambda) \in \mathcal{D}_n(\Lambda')$.  (The point here is that $\varphi$ depends on not only $\Lambda$ and $\Lambda'$ but also the polynomial $P$ under consideration.)  We leave the remaining details to the interested reader.}  As such, we define $\mathcal{D}_n^{\gen(m)}$ to be the set $\mathcal{D}_n(\Lambda)$ for a generic $m$-tuple $\Lambda$.  We can now answer the question from the beginning of this paragraph.
\begin{theorem}\label{thm:generic}
Let $\Lambda$ be an $m$-tuple of distinct complex numbers, and let $n$ be a natural number.  Then $|\mathcal{D}_n(\Lambda)| \leq |\mathcal{D}_n^{\gen(m)}|$, and equality holds if and only if $\mathcal{D}_n(\Lambda)=\mathcal{D}_n^{\gen(m)}$.
\end{theorem}

Finally, we resolve Nathanson's extension problem~\cite{Nathanson}*{Open Question 4}.  An $m \times (n+1)$ $\{0,1\}$-matrix $D$ may or may not be dope; the extension problem for $D$ asks if we can find a polynomial $P$ and an $m$-tuple $\Lambda$ such that $D$ appears in $D_P(\Lambda)$ as the submatrix consisting of the first $n+1$ columns.  We answer this question in the affirmative, and indeed we establish a stronger form in which the tuple $\Lambda$ is also fixed beforehand.
\begin{theorem}\label{thm:extension}
Let $D$ be an $m \times (n+1)$ $\{0,1\}$-matrix, and let $\Lambda$ be an $m$-tuple of distinct complex numbers.  Then there is a polynomial $P$ with $n \leq \deg(P) \leq m(n+2)$ such that $D$ appears in $D_P(\Lambda)$ as the submatrix consisting of the first $n+1$ columns.  Moreover, if all of the entries of $\Lambda$ lie in some subfield $K$ of $\mathbb{C}$, then we may also choose $P$ to have all of its coefficients in $K$.
\end{theorem}
If $D$ is the all-$1$'s matrix, then the ``extending'' polynomial $P$ must have degree at least $m(n+1)$, and in this sense our upper bound on $\deg(P)$ is nearly tight (which answers~\cite{Nathanson}*{Open Question 5}.

We remark that the results of this paper remain valid if $\mathbb{C}$ is replaced by any other field of characteristic $0$.

In Section~\ref{sec:related}, we mention two topics in the existing literature that are related to dope matrices.  In Section~\ref{sec:multiplicity}, we gather several remarks about Nathanson's notion of multiplicity matrices of polynomials, which is essentially equivalent to the notion of dope matrices.  In Section~\ref{sec:two-rows}, we prove Theorem~\ref{thm:m=2} and Corollary~\ref{cor:counting} about $2$-row dope matrices.  In Section~\ref{sec:upper}, we establish the upper bounds Corollary~\ref{cor:counting big} and Theorems~\ref{thm:zero patterns} and~\ref{thm:zero patterns2}.  In Section~\ref{sec:generic}, we prove Theorem~\ref{thm:generic} about generic $m$-tuples $\Lambda$.  In Section~\ref{sec:extension}, we prove Theorem~\ref{thm:extension} about the extension problem, and we show that this result remains close to tight under some natural restrictions.  Finally, in Section~\ref{sec:conclusion}, we present a few open problems.

\section{Some Related Work in the Literature}\label{sec:related}

\subsection{Hermite-Birkhoff Interpolation.}
Fix positive integers $m\le n+1$; a set $E \subseteq [1,m]\times[0,n]$; a complex number $y_{ij}$ for each $(i,j)\in E$; and $m$ distinct complex numbers $\lambda_1,\dots,\lambda_m$.  The framework for Hermite-Birkhoff interpolation asks for conditions under which there is a unique polynomial $P$ of degree at most $n$ satisfying $P^{(j)}(\lambda_i)=y_{ij}$ for all $(i,j) \in E$. The set $E$ is called \emph{poised} if there is such a unique $P$ for every choice of the $y_{ij}$'s. Special cases of this problem include Taylor interpolation, Lagrange interpolation, and Hermite interpolation.

To determine whether a set $E$ is poised, it suffices to consider the case where all $y_{ij}=0$. The zero polynomial certainly satisfies the constraints, and the question becomes whether or not the zero polynomial is the only solution. This differs from our problem on two facets: we insist that our polynomial have degree exactly $n$, and we insist that $P^{(j)}(\lambda_i)\neq 0$ if $(i,j) \in ([1,m]\times[0,n]) \setminus E$.

For the case $m=2$, P\'olya~\cite{Polya} introduced what has since~(see~\cite{Sharma}) come to be called the P\'olya Condition. We invite the reader to compare the following result with our Theorem~\ref{thm:m=2}.

\begin{theorem}[P\'olya~\cite{Polya}]\label{thm:polya}
Let $n\ge1$, and fix a $2\times(n+1)$ $\{0,1\}$-matrix $E=(\epsilon_{ij})_{i \in [1,m], j\in [0,n]}$. Then the following conditions are equivalent:
\begin{itemize}
    \item The zero polynomial is the only polynomial $P$ of degree at most $n$ such that $P^{(j)}(i)=0$ for all pairs $(i,j)$ with $\epsilon_{ij}=1$.
    \item For all  $k\in[1,n+1]$, there are at least $k$ nonzero entries in the first $k$ columns of $E$.
\end{itemize}
\end{theorem}

In fact, P\'olya's result is essentially equivalent to our Theorem~\ref{thm:m=2}, as we will discuss in Section~\ref{sec:two-rows}.

\subsection{Zero Patterns and Sign Patterns}

The bounds for zero patterns of polynomials are similar to earlier
bounds obtained by Warren~\cite{Wa} (see also~\cite{Mi} for 
sign patterns of real polynomials). These supply an upper bound for
the total number of sign patterns of a family of real polynomials
in terms of their degrees and the number of variables. Unlike the
proof in~\cite{RBG}, which is based on simple linear algebra tools,
Warren's work requires techniques from real algebraic geometry.
This result has found a considerable number of combinatorial
applications (see~\cite{Al} and the references therein), and in the context of the present paper it can be used to provide an upper bound
for the number of signed real dope matrices, defined as follows.

\begin{definition}\label{def:signed}
Given a real univariate polynomial $P$ and an $m$-tuple $\Lambda=(\lambda_1, \ldots, \lambda_m)$ of distinct real numbers, the \emph{signed real dope matrix of $P$ with respect to $\Lambda$} is the $m \times (n+1)$ matrix $S_P(\Lambda)=(\sigma_{ij})_{i \in [1,m], j \in [0,n]}$, where $\sigma_{ij}$ equals $0$, $+1$, or $-1$ according to the sign of $P^{(j)}(\lambda_i)$.
\end{definition}
Warren's theorem, together with the reasoning described above, 
implies that the total number of $m \times (n+1)$ 
signed real dope matrices is at most
$$
\left(\frac{Cmn^2}{m+n}\right)^{m+n}
$$
for some absolute constant $C>0$.  This variant could be an interesting topic for future inquiry.

\section{Multiplicity Matrices}\label{sec:multiplicity}
The present work is motivated by the central definition of a recent paper of Nathanson~\cite{Nathanson}.
Given a polynomial $P \in \mathbb{C}[x]$ of degree $n$ and an $m$-tuple $\Lambda=(\lambda_1, \ldots, \lambda_m)$  of distinct complex numbers, Nathanson defines the \emph{multiplicity matrix of $P$ with respect to $\Lambda$} to be the $m \times (n+1)$ matrix $M_P(\Lambda)=(\mu_{ij})_{i \in [1,m], j \in [0,n]}$, with the entry $\mu_{i,j}$ equal to the multiplicity of $\lambda_i$ as a root of $P^{(j)}$.  Nathanson established several properties of these matrices $M_P(\Lambda)$ and closed with eight open questions.

Our setup, which is often more convenient, is equivalent to Nathanson's: $D_P(\lambda)$ is obtained from $M_P(\Lambda)$ by replacing all nonzero entries by $1$'s, and $M_P(\Lambda)$ can be easily recovered from $D_P(\Lambda)$ (see Condition R below).  Nathanson observes that every $m \times (n+1)$ matrix $M_P(\Lambda)$ must satisfy several necessary conditions:
\begin{description}
    \item[Condition E:] The entries $\mu_{ij}$ are nonnegative integers (a condition on \emph{entries}).
    \item[Condition R:] If $\mu_{ij}>0$, then $j<n$ and $\mu_{i,j+1}=\mu_{ij}-1$ (a condition on \emph{rows}).
    \item[Condition C:] For each $j \in [0,n]$, we have $\sum_{i} \mu_{ij} \leq n-j$ (a condition on \emph{columns}).
\end{description}
The necessity of Condition E is obvious, Condition R comes from considering the Taylor expansion of $P$ at each $\lambda_i$, and Condition C is a consequence of the Fundamental Theorem of Algebra.  For $m=1$, Conditions E and R (which here imply Condition C) are also sufficient for a $1 \times (n+1)$ matrix to be the multiplicity matrix of some polynomial.  For $m>1$, these three conditions are not sufficient; to study the $m=2$ case, we introduced a fourth condition (which is necessary only for $m \leq 2)$:
\begin{description}
    \item[Condition T:] For each $k\in[0,n]$, there are at most $k$ nonzero entries in the last $k+1$ columns (a condition on the \emph{tail}).
\end{description}
Note that Condition T is strictly stronger than Condition C. Our Theorem~\ref{thm:m=2} states that Conditions E, R, T are necessary and sufficient for a $2\times(n+1)$ matrix to be a multiplicity matrix. The relevance of Condition T for $m>2$ is uncertain in general, but we conjecture that when $\Lambda$ is generic, Conditions E, R, and T are necessary and sufficient for a $m \times (n+1)$ matrix to be the multiplicity matrix of some polynomial $P$ with respect to $\Lambda$; see Section~\ref{sec:conclusion}.

Finally we mention that, for example,
for (large) $m=n$, almost all matrices satisfying conditions E, R and C are
\emph{not} multiplicity matrices of polynomials.  For simplicity, consider the case where $m=n$ and $n$ is even.  We produce many matrices satisfying Conditions E, R, and C as follows:  Fix a checkerboard pattern of the $n \times (n+1)$ grid; place $0$'s on all of the white squares, and place up to $n-j$ $1$'s on black squares in the $j$-th column for each $j \in [0,n-1]$.  In the $j$-th column, we have
$$\sum_{0 \leq k \leq n-j} \binom{n/2}{k}$$
ways to place the $1$'s.  For $j<(1/4-\varepsilon)n$, this quantity is at least $(1-o(1))2^{n/2}$ (using Chernoff bounds, for instance).  So the number of matrices produced is at least
$$[(1-o(1))2^{n/2}]^{(1/4-\varepsilon)n}=2^{(1-o(1))n^2/8}$$
(sending $\varepsilon$ slowly to zero).  Comparing this rough lower bound with the quantity $2^{4n\log_2(en)}$ from Theorem~\ref{thm:zero patterns} shows that Conditions E, R, and C are nowhere near sufficient.


\section{Dope Matrices with Two Rows}\label{sec:two-rows}
As observed in~\cite{Nathanson}, we can read off the $i$-th row of $D_P(\Lambda)$ from the Taylor expansion of $P$ at $\lambda_i$: for $P(x)=\sum_{j=0}^n a_j (x-\lambda_i)^j$, we have that $\delta_{ij}=1$ if and only if $a_j=0$.  Thus, for a given $m \times (n+1)$ $\{0,1\}$-matrix $D$ and an $m$-tuple $\Lambda$ of distinct complex numbers, deciding whether or not there is some polynomial $P$ with $D=D_P(\Lambda)$ amounts to studying a system of linear constraints on the coefficients of $P$.  The following result of Gessel and Viennot will do the heavy lifting in our proof of Theorem~\ref{thm:m=2}.

\begin{theorem}[\cite{BinomialDeterminants}*{Corollary 2}]\label{thm:GV}
Let ${\cal G},{\cal H}\subseteq\NN\cup\{0\}$ be finite subsets of the same size.  If $\abs{{\cal G}\cap[0,c]}\le \abs{{\cal H} \cap [0,c]}$ for every $c\in\NN\cup\{0\}$, then the matrix of binomial coefficients $$\left( \binom gh \right)_{g\in G, h\in H}$$ is nonsingular.
\end{theorem}

\begin{proof}[Proof of Theorem~\ref{thm:m=2}]
We first show that if $D=(\delta_{ij})$ is a $2 \times (n+1)$ dope matrix, then $D$ satisfies Condition T.  Among all $2 \times (n+1)$ dope matrices that do not satisfy Condition T, choose one such matrix $D=(\delta_{ij})$ that minimizes $n$; we will derive a contradiction.  Assume that $D=D_P(\Lambda)$ for some $P$ and $\Lambda$; as observed in~\cite{Nathanson}, we may assume that $\Lambda=(0,1)$.  The minimality of $n$ implies that
\begin{equation}\label{eq:marginals}
\sum_{j=0}^n (\delta_{1j}+\delta_{2j})=n+1 \quad \text{and} \quad \sum_{j=n-k}^n (\delta_{1j}+\delta_{2j})\leq k \quad \text{for all $k \in [0,n-1]$}.
\end{equation}
In particular, $\delta_{10}=\delta_{20}=1$.  Write
$$P(x)=\sum_{j=0}^n a_j x^j = \sum_{j=0}^n \left(\sum_{k=j}^n \binom kj a_k\right)(x-1)^j,$$
with $a_n \neq 0$ since $P$ has degree $n$.  By the observation at the beginning of this section we have that
\begin{equation}\label{eq:topline}
    a_j = 0 \text{ if and only if } \delta_{1j}=1
\end{equation}
and
\begin{equation}\label{eq:bottomline}
    \sum_{k=j}^n \binom kj a_k = 0 \text{ if and only if } \delta_{2j}=1.
\end{equation}
Let
$${\cal G}:= \{j \in[0,n] : \delta_{1j} = 0\}=\{j : a_j \neq 0\}$$
and
$${\cal H}:=\{j \in[0,n] : \delta_{2j} = 1\} =\{j : \sum_{k=j}^n \binom kj a_k=0\};$$
note that $|\mathcal{G}|=|\mathcal{H}|$ (call this quantity $t$) since $D$ contains the same number of $0$'s and $1$'s (from~\eqref{eq:marginals}).  Write $g_1<g_2<\dots<g_{t}$ and $h_1<\dots<h_{t}$ for the elements of $\mathcal{G}$ and ${\cal H}$, respectively.  Then, after we eliminate the $a_j$'s that we know to equal $0$ from~\eqref{eq:topline}, the constraints in~\eqref{eq:bottomline} are equivalent to the homogeneous system of equations $B\bar a=\bar0$, where $B=(\beta_{kl})_{k,l \in [1,t]}$ is the $t \times t$ matrix with entries $\beta_{kl}=\binom{g_l}{h_k}$ and $\bar a$ is the length-$t$ column vector with entries $a_{j_1}, \ldots, a_{j_t}$.  We will check that $B$ satisfies the conditions of Theorem~\ref{thm:GV}.  For $c\in[0,n-1]$, we have
    \begin{align*}
    \abs{{\cal H} \cap [0,c]} - \abs{{\cal G} \cap [0,c]}
    &= \sum_{j=0}^c \delta_{2j}
    -\sum_{j=0}^c (1-\delta_{1j}) \\
    &=\sum_{j=0}^n (\delta_{1j}+\delta_{2j}) - \left(\sum_{j=c+1}^n (\delta_{1j}+\delta_{2j})\right) -(c+1)\\
    &\geq n+1-(n-(c+1))-(c+1)\\
    &= 1,
\end{align*}
where the inequality uses the latter part of~\eqref{eq:marginals}.  For $c \geq n$, we have
$$    \abs{{\cal H} \cap [0,c]} - \abs{{\cal G} \cap [0,c]}= \sum_{j=0}^n \delta_{2j}
    -\sum_{j=0}^c (1-\delta_{1j})=0,$$
this time using the former part of~\eqref{eq:marginals}.  So we can apply Theorem~\ref{thm:GV} (in fact with room to spare) to the system $B\bar a=\bar 0$, and we conclude that $\bar a=\bar 0$.  But the last entry is $a_{j_t}=a_n$, and this contradicts the assumption that $a_n \neq 0$.

We now show that if $D$ is a $2 \times (n+1)$ $\{0,1\}$-matrix satisfying Condition T, then $D$ is dope.  Let the total number of $1$'s in $D$ be $n-h$, where we know that $h \geq 0$.  Form the $2 \times (n+h)$ matrix $D'$ by prepending $h$ columns of all $1$'s to $D$; note that $D'$ still satisfies Condition T and exactly $n+h$ entries of $D'$ are $1$'s.  Proceeding as in the first half of the proof, we find that there is a (unique, in fact) monic polynomial $P$ of degree $n+h$ such that $D_P(0,1)$ has $1$'s in all of the positions where $D'$ has $1$'s.  Since $D_P(0,1)$ must satisfy Condition T (by the first half of the proof), we see that at most $n+h$ entries of $D_P(0,1)$ equal $1$.  Since $n+h$ entries of $D'$ equal $1$, we conclude that $D_P(0,1)=D'$.  Thus, $D_{P^{(h)}}(0,1)=D$, as desired.
\end{proof}

In the preceding proof, Theorem~\ref{thm:GV} of Gessel and Viennot can be replaced by an application of P\'olya's Theorem~\ref{thm:polya}.  Indeed, one can check that the constraints in Equations~\ref{eq:topline} and~\ref{eq:bottomline} satisfy P\'olya's Condition, and then Theorem~\ref{thm:polya} will give the desired contradiction that $\overline{a}=\overline{0}$; the application of Theorem~\ref{thm:polya} in place of Theorem~\ref{thm:GV} in the second half of the proof is analogous.  We also mention that Theorem~\ref{thm:polya} can be quickly deduced from the proof of Theorem~\ref{thm:m=2} (with Theorem~\ref{thm:GV} as the key input).

We now prove the enumerative results in Corollary~\ref{cor:counting}.  Let $C(n,t)$ denote the number of dope $2 \times (n+1)$ matrices with exactly $t$ entries equal to $1$.

\begin{proposition}\label{prop:refined-count}
For $n\geq 1$ and $0\leq t \leq n+1$, we have
$$C(n,t)=\binom{2n+1}{t}-\binom{2n+1}{t-1}.$$
\end{proposition}
\begin{proof}
We have $C(n,-1)=0$, $C(n,0)=1$ (the all-$0$'s matrix), and $C(0,t)=0$ for $t>0$. The result now follows from the recurrence
$$  C(n,t)=C(n-1,t)+2C(n-1,t-1)+C(n-1,t-2),$$
which comes from conditioning on the first column.
\end{proof}

Let $B(n,2s)$ denote the number of dope $2 \times (n+1)$ matrices with exactly $2s$ entries equal to $1$ and such that the two rows are identical.

\begin{proposition}\label{prop:B}
For $n \geq 1$ and $0 \leq 2s \leq n+1$, we have
$$B(n,2s)=\binom{n}{s}-\binom{n}{s-1}.$$
\end{proposition}

\begin{proof}
We have $B(n,0)=1$ (the all-$0$'s matrix) and $B(0,2s)=0$ for $s>0$.  The result now follows from the recurrence $$B(n,2s)=B(n-1,2s)+B(n-1,2s-2),$$
which comes from conditioning on the first column.
\end{proof}

We are ready to prove Corollary~\ref{cor:counting}.

\begin{proof}[Proof of Corollary~\ref{cor:counting}]
The first statement of the corollary follows from Proposition~\ref{prop:refined-count} by telescoping.  The second statement is precisely from the $t=n$ case of Proposition~\ref{prop:refined-count}.  For the third statement, note that the number of $2 \times (n+1)$ dope matrices modulo the action of swapping the two rows is
$$\frac{1}{2} \bigg[\sum_t C(n,t)+\sum_s B(n,2s) \bigg];$$
substituting from Propositions~\ref{prop:refined-count} and~\ref{prop:B} and collapsing the telescoping sums gives the result.
\end{proof}

\section{Upper Bounds}\label{sec:upper}
We now turn to upper bounds on the number of dope matrices.  We begin with Corollary~\ref{cor:counting big}, which is a consequence of Corollary~\ref{cor:counting}.

\begin{proof}[Proof of Corollary~\ref{cor:counting big}]
Note that each pair of rows of a dope matrix forms a $2$-row dope matrix; Corollary~\ref{cor:counting} tells us that there are at most $\binom{2n+1}{n}$ possibilities for such a $2$-row dope matrix.  We wish to apply Shearer's Inequality (see~\cite{AlonSpencer}*{Proposition 15.7.5}).  Let $X=\{0,1\}^{n+1}$, and consider a $m \times (n+1)$ dope matrix $D$ as an element of $X^m$, where the $i$-th row of $D$ corresponds to the $i$-th copy of $X$; in this way we also consider $\mathcal{D}_n^m$ as a subset of $X^m$.  For each choice of $1 \leq i<i' \leq m$, the projection of $\mathcal{D}_m^n$ onto the $i$-th and $i'$-th copies of $X$ has size at most $\binom{2n+1}{n}$.  By Shearer's Inequality applied to all such pairs $(i, i')$, we find that
$$|\mathcal{D}_m^n|^{m-1} \leq \prod_{1 \leq i <i' \leq m}\binom{2n+1}{n}=\binom{2n+1}{n}^{m(m-1)/2}$$
and hence
$$|\mathcal{D}_m^n| \leq \binom{2n+1}{n}^{m/2},$$
as desired.  The second estimate in the Corollary follows from plugging in Robbins' factorial bound~\cite{Robbins}.
\end{proof}

We now turn to Theorems~\ref{thm:zero patterns} and~\ref{thm:zero patterns2}.  If $f=(f_1, \ldots, f_a)$ is a sequence of polynomials in $b$ variables over the field $K$, then the \emph{zero pattern} of $f$ evaluated at the point $u \in K^b$ is the set 
    $$Z_f(u)=\{i \in [1,a]: f_i(u)=0\}.$$  
We let $Z_f$ denote the total number of distinct zero patterns that appear as $u$ ranges over $K^b$.  We will use the following result of R\'onyai, Babai, and Ganapathy~\cite{RBG}.

\begin{theorem}[\cite{RBG}*{Theorem 1.1}]\label{thm:RBG}
Let $f=(f_1, \ldots, f_a)$ be a sequence of polynomials in $b$ variables over the field $K$, and let $d_i$ denote the degree of $f_i$.  Then $$Z_f \leq \binom{b+\sum_{i=1}^a d_i}{b}.$$
\end{theorem}

Translating from the language of dope matrices to the language of zero patterns gives Theorems~\ref{thm:zero patterns} and~\ref{thm:zero patterns2}, as follows.
\begin{proof}[Proof of Theorem~\ref{thm:zero patterns}]
Let $K=\mathbb{C}$.  Note that the last column of a dope matrix always has all $0$'s.  Note also that we still obtain all of the dope matrices if we restrict our attention to monic polynomials and to tuples whose first two entries are $0,1$.  Let $P(x)=x^n+\sum_{j=0}^{n-1} a_j x^j$ and $\Lambda=(0,1,\lambda_3, \ldots, \lambda_m)$ be the general forms of such a monic polynomial of degree $n$ and an $m$-tuple starting with $0,1$, respectively.  Consider $$f=(P^{(j)}(\lambda_i))_{i \in [1,m], j \in [0,n-1]}$$
as an $mn$-tuple of polynomials in the $m+n-2$ variables $a_0, \ldots, a_{n-1}, \lambda_3, \ldots, \lambda_m$.  Then for each choice of $u=(a_0, \ldots, a_{n-1}, \lambda_3, \ldots, \lambda_m)$, the dope matrix $D_P(\Lambda)$ is determined by the zero pattern $Z_f(u)$.  In particular, the number of $m \times (n+1)$ dope matrices is at most $Z_f$.  For $i \leq 2$, the polynomial $f^{(j)}(\lambda_i)$ has degree $1$.  For $i>2$, the polynomial $f^{(j)}(\lambda_i)$ has degree $n-j$.  So the sum of the degrees of the polynomials in $f$ is
$$2n+(m-2)\sum_{j=0}^{n-1}(n-j) =2n+\frac{(m-2)n(n+1)}{2}.$$
Now Theorem~\ref{thm:RBG} gives $Z_f \leq \binom{3n+m-2+(m-2)n(n+1)/2}{m+n-2}$, as desired.
\end{proof}

We remark that the number of dope matrices is in fact slightly smaller than $Z_f$ (in the notation of the above proof) because of the constraint that the $\lambda_i$'s be distinct.  It does not appear that this observation yields a significant improvement to Theorem~\ref{thm:zero patterns}.  
For Theorem~\ref{thm:zero patterns2} (counting dope matrices with $\Lambda$ fixed), we will use the following strengthening of Theorem~\ref{thm:RBG} for the case where the polynomials $f_i$ are linear.

\begin{theorem}[\cite{RBG}*{Theorem 1.2}]\label{thm:RBG2}
Let $f=(f_1, \ldots, f_a)$ be a sequence of linear polynomials in $b$ variables over the field $K$.  Then
$$Z_f \leq \sum_{i=0}^b \binom{a}{i}.$$
\end{theorem}

\begin{proof}[Proof of Theorem~\ref{thm:zero patterns2}]
We proceed as in the proof of Theorem~\ref{thm:zero patterns}, but this time we consider $P^{(j)}(\lambda_i)$ as a polynomial in only the $n$ variables $a_0, \ldots, a_{n-1}$ (since the $\lambda_i$'s are fixed).  Then each $P^{(j)}(\lambda_i)$ is linear, and Theorem~\ref{thm:RBG2} gives $Z_f \leq \sum_{i=0}^n \binom{mn}{i}$  The last estimate in the Theorem follows from standard entropy bounds on sums of binomial coefficients (e.g.,~\cite{AlonSpencer}*{Corollary 15.7.3}).
\end{proof}

\section{Generic $\Lambda$}\label{sec:generic}

We require some setup before we prove Theorem~\ref{thm:generic}.  Fix a tuple $\Lambda=(\lambda_1, \ldots, \lambda_m)$ and a natural number $n$.  Consider $P(x)=a_nx^n+\cdots+a_0$ as the general form of a polynomial of degree $n$, where the $a_k$'s are viewed as variables.  We wish to study how the $m \times (n+1)$ dope matrix $D_P(\Lambda)$ depends on the choice of the $a_k$'s.  To each element $(i,j) \in [1,m] \times [0,n]$, we associate the linear form $L_{i,j}$ in the variables $a_0, \ldots, a_n$ given by
$$\sum_{k=j}^n \binom{k}{j} \lambda_i^{k-j} a_k.$$
Notice that $P^{(j)}(\lambda_i)$ vanishes if and only if $L_{i,j}=0$; in particular, in the dope matrix $D_P(\Lambda)$, the set of positions with $1$'s is precisely
$$\{(i,j) \in [1,m] \times [0,n]: L_{i,j}=0\}.$$
We also view $L_{i,j}$ as an element of $\mathbb{C}^{n+1}$ with entries $\left(\binom{k}{j} \lambda_i^{k-j} \right)_{0 \leq k \leq n}$ (note that $\binom{k}{j}$ vanishes for $k<j$ and hence is harmless).

Let $D$ be a $m \times (n+1)$ $\{0,1\}$-matrix, and let $E_D \subseteq [1,m] \times [0,n]$ denote the set of positions of $D$ with $1$'s.  Let $v_0=(0, \ldots, 0, 1) \in \mathbb{C}^{n+1}$.  Then, by the preceding discussion, there exists a degree-$n$ polynomial $P$ satisfying $D=D_P(\Lambda)$ if and only if there is a subspace $V$ of $\mathbb{C}^{n+1}$ such that $v_0 \notin V$ and
$$\{(i,j) \in [1,m] \times [0,n]: L_{i,j} \in V\}=E_D.$$

(The condition $v_0\notin V$ corresponds to the requirement that $P$ has degree exactly $n$.)  If such a $V$ exists, we may of course take it to equal $\Span\{L_{i,j}: (i,j) \in E_D\}$; denote this subspace by $V_D$.  We are now ready to prove Theorem~\ref{thm:generic}.

\begin{proof}[Proof of Theorem~\ref{thm:generic}]
For each $D \in \mathcal{D}_n(\Lambda)$, consider the associated subspace $V_D$, and fix some $F_D \subseteq E_D$ so that $\{L_{i,j}: (i,j) \in F _D\}$ is a basis for $V_D$.  It follows that $\{L_{i,j}: (i,j) \in F _D\} \cup \{v_0\}$ is linearly independent.  Let $\Lambda'=(\lambda'_1, \ldots, \lambda'_m)$ be a generic $m$-tuple, with associated linear forms $L'_{i,j}$.

We claim that the set
$$\{L'_{i,j}: (i,j) \in F_D\} \cup \{v_0\}$$
is also linearly independent.  To see this, consider $\{L_{i,j}: (i,j) \in F_D\} \cup \{v_0\}$ as the rows of a $(\dim(V)+1)\times(n+1)$ matrix $M$, and consider $\{L'_{i,j}: (i,j) \in F_D\} \cup \{v_0\}$ as the rows of a $(\dim(V)+1)\times(n+1)$ matrix $M'$.  The linear independence of $\{L_{i,j}: (i,j) \in F_D\} \cup \{v_0\}$ guarantees that $M$ has a $(\dim(V)+1) \times (\dim(V)+1)$ submatrix $N$ with full rank, i.e., nonzero determinant.  Let $N'$ be the corresponding submatrix of $M'$.  It clearly suffices to show that $\det(N') \neq 0$.  Suppose instead that $\det(N')=0$.  Since the $\lambda'_i$'s are algebraically independent, we may view this equation as an equality of polynomials in the variables $\lambda'_1, \ldots, \lambda'_m$; in particular, $\det(N')$ is the zero polynomial.  This implies, upon replacing the $\lambda'_i$'s by the corresponding $\lambda_i$'s, that $\det(N)=0$, and this contradiction establishes the claim.

Let $$V'_D=\Span\{L'_{i,j}: (i,j) \in F_D\}.$$  The previous claim shows that $v_0 \notin V'_D$ and that $\{L'_{i,j}: (i,j) \in F_D\}$ is a basis.  Let $$E'_D=\{(i,j): L'_{i,j} \in V'_D\},$$ and define $D' \in \mathcal{D}_n(\Lambda')$ to be the dope matrix with $1$'s in the positions $E'_D$.  This procedure gives us an (admittedly non-canonical) map $\psi: \mathcal{D}_n(\Lambda) \to \mathcal{D}_n(\Lambda')$ via $D \mapsto D'$; it remains to show that $\psi$ is injective.  Suppose $\psi(D)=\psi(\widetilde{D})$, i.e., $V'_D=V'_{\widetilde{D}}$.  For each $(i_0,j_0) \in F_D$, we can express $L'_{i_0,j_0}$ as a linear combination of $\{L'_{i,j}: (i,j)\in F_{\widetilde{D}}\}$.  We claim that then $L_{i_0,j_0}$ is a linear combination of $\{L_{i,j}: (i,j)\in F_{\widetilde{D}}\}$.  Indeed, if this were not the case, then the (multi-)set $$\{L_{i,j}: (i,j)\in F_{\widetilde{D}}\} \cup \{L_{i_0,j_0}\}$$ would be linearly independent; viewing $\{L_{i,j}: (i,j)\in F_{\widetilde{D}}\} \cup \{L_{i_0,j_0}\}$ as the rows of a $(|F_{\widetilde{{D}}}|+1) \times (n+1)$ matrix $M$, we can then find a $(|F_{\widetilde{{D}}}|+1) \times (|F_{\widetilde{{D}}}|+1)$ submatrix $N$ with full rank, i.e., nonzero determinant.  Let $M'$ be the $(|F_{\widetilde{{D}}}|+1) \times (n+1)$ matrix with rows given by $\{L'_{i,j}: (i,j)\in F_{\widetilde{D}}\} \cup \{L'_{i_0,j_0}\}$, and let $N'$ be the corresponding submatrix of $M'$.  By the argument in the previous paragraph, we deduce $\det(N') \neq 0$ from $\det(N) \neq 0$.  But this contradicts $L'_{i_0,j_0}$ being a linear combination of $\{L'_{i,j}: (i,j)\in F_{\widetilde{D}}\}$, which establishes the claim.  It follows that $V_D \subseteq V_{\widetilde{D}}$.  The same argument gives $V_{\widetilde{D}}\subseteq V_D$, and hence $V_D=V_{\widetilde{D}}$.  So $D=\widetilde{D}$, which establishes the injectivity of $\psi$.

It remains only to characterize equality cases.  Suppose $|\mathcal{D}_n(\Lambda)|=|\mathcal{D}_n^{\gen(m)}|$.  Then the map $\psi$ is a bijection.  Moreover, for each $D \in \mathcal{D}_n(\Lambda)$, the corresponding matrix $D'=\psi(D)$ is independent of the choice of $F_D$.  We will show that in fact $D=D'$.  Since each element of $E_D$ is part of a basis for $V_D$, we see that $E_{D} \subseteq E_{D'}$ (where $E_{D'}=E'_D$ unambiguously because this doesn't depend on the choice of $F_D$).  For the reverse inclusion, let $(i_0, j_0) \in E_{D'}$ be any element, and fix some choice of $F_D$.  Then $L'_{i_0, j_0}$ is a linear combination of $\{L'_{i,j}: (i,j) \in F_D\}$.  By the argument in the previous paragraph, we see that also $L_{i_0, j_0}$ is a linear combination of $\{L_{i,j}: (i,j) \in F_D\}$, and thus $(i,j) \in E_D$.  So $E_{D'}\subseteq E_{D}$ and hence $E_D=E_{D'}$; we conclude that $D=D'$, as desired.  This completes the proof.
\end{proof}

\section{The Extension Problem}\label{sec:extension}
Our resolution of the extension problem follows from a quick application of the Chinese Remainder Theorem for polynomial rings.

\begin{proof}[Proof of Theorem~\ref{thm:extension}]
Write $D=(\delta_{ij})$ and $\Lambda=(\lambda_1, \ldots, \lambda_m)$.  For each $1 \leq i \leq m$, define the set
$$A_i=\{j \in [0,n]: \delta_{ij}=0\}$$
and the polynomial
$$P_i(x)=\sum_{j \in A_i}(x-\lambda_i)^j.$$
By the observation at the beginning of Section~\ref{sec:two-rows}, it suffices to find a polynomial $P$ of degree at least $n$ such that $P \equiv P_i \pmod{(x-\lambda_i)^{n+2}}$ for every $i$.  By the Chinese Remainder Theorem (see, e.g.,~\cite{AT}*{Proposition 1.10}), we can find such a $P$ with degree at most $m(n+2)$.  If $P$ has degree at least $n$, then we are done; if $P$ has degree strictly smaller than $n$, then we can ensure that $P$ has degree exactly $m(n+2)$ by adding $\prod_{i=1}^m (x-\lambda_i)^{n+2}$ to $P$.  Either way, $D$ appears as the left-most portion of $D_P(\Lambda)$, as desired.  The ``moreover'' statement is clear from the proof.
\end{proof}

The computation at the end of Section~\ref{sec:multiplicity} shows that Theorem~\ref{thm:extension} remains close to tight even if we require $D$ to ``come from'' a matrix $M$ satisfying Conditions E, R, and C (as suggested in Nathanson's original question); call such a $\{0,1\}$-matrix an \emph{ERC-matrix}.  Suppose that for all $n \times (n+1)$ ERC-matrices, we can solve the extension problem with a polynomial of degree at most $N$.  Write $P(x)=a_Nx^N+\cdots+a_0$, where now we allow $a_N$ to be $0$ (so that we also capture polynomials of degree smaller than $N$).  Since each $n \times (n+1)$ ERC matrix has all $0$'s in the last column, it is completely determined by its left-most $n \times n$ block.  By Theorem~\ref{thm:RBG2}, the number of possibilities for the left-most $n \times n$ block of $D_P(\Lambda)$ is at most
$$\sum_{i=0}^{N+1}\binom{n^2}{i} \leq 2^{n^2 H((N+1)/n^2)}$$
(see the proof of Theorem~\ref{thm:zero patterns2}).  Comparison with the computation of Section~\ref{sec:multiplicity} gives that 
$$2^{(1-o(1))n^2/8} \leq 2^{n^2 H((N+1)/n^2)},$$
and the monotonicity of $H(p)$ for $0<p<1/2$ implies that
$$N \geq (0.017-o(1))n^2;$$
this shows that even in this setting Theorem~\ref{thm:extension} is tight up to a constant factor.

\section{Open Problems}\label{sec:conclusion}
Following Nathanson's lead, we close with several questions that we haven't resolved.  Nathanson characterized $\mathcal{D}_n(\lambda)$ for all complex numbers $\lambda$, and in Theorem~\ref{thm:m=2} we characterized $\mathcal{D}_n(\lambda_1, \lambda_2)$ for all pairs of distinct complex numbers $\lambda_1, \lambda_2$.  In each of these settings, the particular choice of $\Lambda$ is immaterial since all tuples $\Lambda$ are generic.  When $\Lambda$ is an $m$-tuple for $m \geq 3$, however, the set $\mathcal{D}_n(\Lambda)$ does depend on the particular choice of $\Lambda$, since not all triples are generic.  Although it seems hopeless to characterize $\mathcal{D}_n(\Lambda)$ for all $\Lambda$, we offer the following simple-to-state-conjecture for the case where $\Lambda$ is generic.

\begin{conjecture}\label{conj:generic}
Let $m \geq 3$.  Then $\mathcal{D}_n^{\gen(m)}$ consists of exactly the $m \times (n+1)$ $\{0,1\}$-matrices satisfying the following condition: For each $k \in [0,n]$, there are at most $k$ nonzero entries in the last $k+1$ columns.
\end{conjecture}
We remark that the $m=1$ version of this conjecture holds trivially and that the $m=2$ version is the content of Theorem~\ref{thm:m=2}.  Also, explicit computations of $\mathcal{D}_n^{\gen(m)}$ for $m=3$ and $n \leq 6$ are consistent with Conjecture~\ref{conj:generic}.  Finally, since Theorem~\ref{thm:generic} is a fairly quick consequence of Conjecture~\ref{conj:generic}, we may regard the former as evidence in support of the latter.  If Conjecture~\ref{conj:generic} is true, then it also gives a simple linear recurrence relation for $|\mathcal{D}_n^{\gen(m)}|$ for each fixed $m$ (by conditioning on the left-most column, as in the second half of Section~\ref{sec:two-rows}).

It could also be of interest to count or estimate $|\mathcal{D}_n(\Lambda)|$ for other particular choices of $\Lambda$, such as $\Lambda=(0,1, \ldots, m-1)$.  Furthermore, one could study the full set $\mathcal{D}_n^m$ of $m \times (n+1)$ dope matrices.  The most natural open problem about $\mathcal{D}_n^m$ is enumerative.

\begin{problem}\label{problem}
Estimate $|\mathcal{D}_n^m|$ for various values of the parameters $m$ and $n$.
\end{problem}

In a related direction, we can define an equivalence relation $\sim$ on $m$-tuples of distinct complex numbers by saying that $\Lambda \sim \Lambda'$ if $\mathcal{D}_n(\Lambda)=\mathcal{D}_n(\Lambda')$ for all natural numbers $n$.  We remarked in the introduction that $\Lambda\sim \Lambda'$ if $\Lambda'=a+b\varphi(\Lambda)$ for some $a \in \mathbb{C}$, $b \in \mathbb{C}^{\times}$, and $\varphi$ a $\mathbb{Q}$-automorphism of $\mathbb{C}$.  It would be interesting to characterize the equivalence classes of $\sim$ and, in particular, to determine whether or not all equivalent tuples are related as described in the previous sentence.

We highlight that this problem is nontrivial even for the case $m=3$.  Here, every triple is equivalent to $(0,1,\lambda)$ for some $\lambda$ (not $0$ or $1$).  It is clear from the preceding discussion that $(0,1,\lambda) \sim (0,1,\lambda')$ if either: $\lambda$ and $\lambda'$ are both transcendental; or $\lambda$ and $\lambda'$ are conjugate algebraic numbers.  We conjecture that all equivalences are of this form.

\begin{conjecture}\label{conj:equiv}
If $(0,1,\lambda) \sim (0,1,\lambda')$, then either: $\lambda$ and $\lambda'$ are both transcendental; or $\lambda$ and $\lambda'$ are conjugate algebraic numbers.
\end{conjecture}
One can easily construct examples showing that if $\lambda$ is a rational number other than $0$ or $1$, then $(0,1,\lambda)$ is not equivalent to any $(0,1,\lambda')$ with $\lambda \neq \lambda'$; this is the ``degree-$1$'' case of the above conjecture.  Even in the ``degree-$2$'' case, it is not clear how to prove the conjecture for the case where, for instance, $\lambda$ is the square root of a positive rational.

There are also dope-matrix-related objects that could be interesting to study.  First, we reiterate that signed dope matrices (see Definition~\ref{def:signed}) are a promising topic for future research.  Second, we remark that the discussion of Section~\ref{sec:generic} gives a natural way to associate a (representable) matroid $M_n(\Lambda)$ with each $m$-tuple $\Lambda$ and natural number $n$: The ground set is $[1,m] \times [0,n]$, and a subset of the ground set is independent if the corresponding collection of linear forms $L_{i,j}$ is linearly independent.  The $m \times (n+1)$ dope matrices arising from $\Lambda$ correspond to the flats of $M_n(\Lambda)$ that do not contain $L_{1,n}$ (equivalently, the proper flats of the contraction $M_n(\Lambda)/(1,n)$).  What can one say about the class of all matroids $M_n(\Lambda)$?

Finally, we provide the following data on $|\mathcal{D}_n(\Lambda)|$ for several triples $\Lambda$; we hope that patterns in this data will inspire future work.
    $$\begin{array}{c|ccccccc}
     & \multicolumn{7}{c}{n} \\
  \Lambda & 0 & 1 & 2 & 3 & 4 & 5 & 6 \\ [0.5ex]
\hline
 (0,1,\sqrt{2} ) & 1 & 4 & 19 & 98 & 525 & 2944 & 16870 \\
 (0,1,2) & 1 & 4 & 17 & 86 & 446 & 2503 & 14228 \\
 (0,1,3) & 1 & 4 & 19 & 92 & 497 & 2793 & 16022 \\
 (0,1,\pi) & 1 & 4 & 19 & 98 & 531 & 2974 & 17060 \\
 (0,1,4) & 1 & 4 & 19 & 98 & 519 & 2911 & 16712 \\
\end{array}$$
The OEIS~\cite{oeis} recognizes these terms of $|\mathcal{D}_n(0,1,\pi)|$ as coinciding with sequence A047099, which is the sequence satisfying the linear recurrence suggested by Conjecture~\ref{conj:generic}.

\emph{Note added in revision}: Since this paper appeared as a preprint, Bisain~\cite{ankit} has proven Conjecture~\ref{conj:generic}, provided fairly precise bounds for Problem~\ref{problem}, and derived several other interesting results on dope matrices.

\section*{Acknowledgements}
We thank Melvyn B. Nathanson for helpful conversations and for introducing us to this problem.  We are also grateful to Moshe Newman for suggesting the analogy with Birkhoff Interpolation, and we benefited from conversations with Sergei Konyagin.  The Online Encyclopedia of Integer Sequences~\cite{oeis} was helpful at several stages of this work.  Finally, we thank an anonymous referee for their careful reading of our paper.

The first author is supported in part by NSF grant DMS--1855464 and BSF grant 2018267. The second author is supported by an NSF Graduate Research Fellowship (grant DGE--2039656).  The third author is supported by a PSC--CUNY Award, jointly funded by The Professional Staff Congress and The City University of New York.

\end{document}